\documentclass[a4paper,12pt]{article}
\usepackage{color}
\usepackage{array}
\usepackage{subfigure}

\usepackage{amsfonts}
\usepackage{amsmath}
\usepackage{amsthm}
\usepackage{graphicx}
\usepackage{float}

\usepackage[british,UKenglish,USenglish,english,american]{babel}

\usepackage{color}
\numberwithin{equation}{section}

\theoremstyle{remark}\newtheorem{note}{Note}[section]
\theoremstyle{plain}\newtheorem{teo}{Theorem}[section]
\theoremstyle{definition}\newtheorem*{defi}{Definition}
\theoremstyle{plain}
\theoremstyle{plain}\newtheorem*{propo1}{Proposition i}
\theoremstyle{plain}\newtheorem*{propo2}{Proposition ii}
\theoremstyle{plain}\newtheorem*{propo3}{Proposition iii}
\theoremstyle{plain}\newtheorem*{propo4}{Proposition iv}

\theoremstyle{plain}\newtheorem{lem}{Lemma}[section]
\theoremstyle{plain}
\theoremstyle{plain}

\theoremstyle{remark}\newtheorem{obs}{Remark}[section]

\newcommand{\bbR}{{\mathbb R}}
\newcommand{\RR}{{\mathbb R}}
\newcommand{\calA}{{\mathcal A}}

\title{Fully discrete schemes for monotone optimal control problems}
\author{Eduardo A. Philipp $^\dag$ \and Laura S. Aragone $^\dag$ \and Lisandro A. Parente $^\dag$}
\date{\em{\small{$\dag$: CIFASIS - CONICET - UNR - Argentina \\ e-mails: philipp@cifasis-conicet.gov.ar; aragone@cifasis-conicet.gov.ar; parente@cifasis-conicet.gov.ar}}}
\begin{document}
\maketitle
\begin{abstract}
In this article we study a finite horizon optimal control problem with monotone controls. We consider the associated Hamilton-Jacobi-Bellman (HJB) equation which characterizes the value function.

We consider the totally discretized problem by using the finite
element method to approximate the state space $\Omega$. The obtained problem is
equivalent to the resolution of a finite sequence of stopping-time problems.

The convergence orders of these approximations are proved, which are
in general $(h+\frac{k}{\sqrt{h}})^\gamma$ where $\gamma$ is the H\"older constant of the value function $u$. A special election of the
relations between the parameters $h$ and $k$ allows to obtain a
convergence of order $k^{\frac{2}{3}\gamma}$, which is valid without
semiconcavity hypotheses over the problem's data.

We show also some numerical implementations in an example.
\end{abstract}

\section{Introduction}

We study a fully discrete scheme for the numerical resolution of the
infinite horizon monotone optimal control problem, through the
analysis of the associated finite horizon problem as in \cite{APP}.

The considered system is governed by the following differential
equation
\begin{equation}
\left\{ \begin{array}{ll} \dot{y}(s)=g(y(s),\alpha(s)), & s>0, \\
y(0)=x, & \end{array} \right. \label{dinam}
\end{equation}
where $\alpha(\cdot)$ is the control function, which throughout this
article is restrained to be a non-decreasing monotone function
defined in the set $[0,+\infty)$, with values in $[0,1]$. For each
$a \in [0,1]$ let $\calA(a)$ be the set of these functions with
initial values higher or equal to $a$, i.e. $\alpha(0) \ge a$. At a
point $s>0$, the vector $y(s) \in \bbR^\nu$ is the state
corresponding to the control $\alpha(\cdot)$ employed; $x \in
\Omega$ is the initial state, where $\Omega \subset \bbR^\nu$ is an open
set. We assume that the evolution of the system always stays in
$\Omega$, no matter which control is selected.

The performance of the employed control is measured through the
functional $J$
$$ J(x,\alpha(\cdot)):= \int\limits_0^\infty f(y(s),\alpha (s))e^{-\lambda s} \text{d} s, $$
where $f: \Omega \times [0,1] \rightarrow \bbR$ is the instantaneous
cost and $\lambda > 0$ is the discount factor. The problem consists
in finding for each pair $(x,a) \in \Omega \times [0,1]$, a control
$\bar{\alpha} \in \calA(a)$ that attains the minimum of the
functional $J$. Therefore, the value function $u$ is
\begin{equation}
u(x,a):= \inf_{\alpha \in \calA(a)} J(x,\alpha(\cdot)), \label{defu}
\end{equation}
which allows to construct optimal or suboptimal policies in feedback
\cite{13,14}. \\
We assume the following Lipschitz continuity and boundedness hypotheses on
the functions $g$ and $f$: there exist positive constants $L_g, M_g,
L_f$ and $M_f$ such that $\forall\, x,\bar{x} \in \Omega,\;
\forall\, a,\bar a \in [0,1],$

\begin{equation}
 \| g(x,a)-g(\bar{x},\bar a) \|  \le  L_g \left( \|x-\bar{x}\| + |a-\bar a| \right), \hspace{12mm}
\| g(x,a) \| \le M_g,   \label{condg}
\end{equation}

\begin{equation}
| f(x,a)-f(\bar{x},\bar a) |  \le  L_f \left(\|x-\bar{x}\| + |a-\bar
a| \right), \hspace{13mm} | f(x,a) | \le  M_f. \label{condf}
\end{equation}
Under these assumptions, using classic arguments  it can be proven
that the function $u$ is bounded and H\"older continuous in both
variables with H\"older constant $\gamma$:
\begin{equation}
\left\{ \begin{array}{rcl} \gamma = 1 & \text{ if} & \lambda > L_g  \vspace{2mm}\\
\gamma = \frac{\lambda}{L_g}  & \text{ if} & \lambda < L_g  \vspace{2mm}\\
\gamma \in (0,1) & \text{ if} & \lambda = L_g  \end{array} \right.
\end{equation}
We consider the quasi-variational HJB inequality associated to the
problem, given by the equation
\begin{equation}
\min \left( Lu(x,a),\frac{\partial u(x,a)}{\partial a} \right) = 0,
\qquad \text{in } \Omega \times (0,1), \label{HJB}
\end{equation}
where
\[ Lu(x,a) = \frac{\partial u(x,a)}{\partial x} g(x,a) + f(x,a) - \lambda u(x,a), \]
and the boundary condition
\begin{equation}
u(x,1) = \int\limits_0^\infty f(\eta(s),1)e^{-\lambda s} \text{d}s,
\label{bondcond}
\end{equation}
where $\eta(s)$ is the trajectory corresponding to the control $\alpha \equiv 1$, and  $x$ is the initial value.\\
In general, this equation does not admit a solution in $C^1(\Omega
\times (0,1))$, so the notion of viscosity solutions comes into play
(see \cite{37}). Specifically,  $u$ is the unique viscosity
solution of the  HJB equation \eqref{HJB} with boundary conditions
\eqref{bondcond} (see \cite{13}).
\subsection{Discretization in time and control space}
We would like to obtain discretization schemes in order to
numerically solve the equation \eqref{HJB}. For this aim, we
introduce an auxiliary optimal control problem where the policies
have the additional restriction of being uniformly step functions
with values in a discrete, equi-spaced set. Specifically, the
control variable $a$ takes values in the set
\[
I_h := \left\{ ih | i=0 \dots \frac{1}{h} \right\}.
\]
We also define $I_h(a) = I_h \cap [a,1]$. In the space $C(\Omega
\times I_h)$ we consider the operators
\begin{eqnarray}\nonumber
\left( A^{h,b}(w) \right)(x,a) &=& (1-\lambda h) w(x+h g(x,a),b) + h f(x,a),\\
\left(A^h(w) \right) (x,a) &=& \min_{b \in I_h(a)} \left( A^{h,b}(w)
\right) (x,a), \label{esq1}
\end{eqnarray}
arising from the discretization of the HJB equation.

Having considered a discretization in time, we introduce the
consistent problem of finding in the functional space $C(\Omega
\times I_h)$
\begin{equation}
\text{Problem } P^h: \text{ Find the fixed point } u^h \text{ of the
operator } A^h \label{ptofijo}
\end{equation}
From \eqref{esq1} it turns out that $A^h$ is a contractive operator
if $0<h<\frac{1}{\lambda}$. From this property it is straightforward
to prove that \eqref{ptofijo} has an unique solution $u^h$, which is
bounded and uniformly H\"{o}lder continuous in the first variable (see
\cite{27}). If \eqref{condg} and \eqref{condf} are satisfied, the
same techniques give the H\"older continuity in both variables.


\subsection{Fully discrete infinite horizon problem}
The previously presented discretization allowed us to obtain important convergence results (see \cite{APP}). Nevertheless, in order to obtain numerical methods to estimate $u^h$ it is also necessary the discretization in the
state variables. We obtain a solution $u^h_k$ totally discrete, considering a
discretization in these variables by the finite element method.
\subsubsection*{Discretization elements}
Let $\{S_j^k\}$ be a family of triangulations of $\Omega$, i.e. a set of simplices that approximate $\Omega$ in the following sense: \vspace{5mm} \\
$\Omega_k = \bigcup_j S_j^k$ is a polyhedron of $\mathbb{R}^\nu$
such that the following properties hold: \begin{equation} \max_j
(\text{diam } S_j^k) = k. \label{HIP1} \end{equation}
\begin{equation} \exists h_0
> 0 \text{ such that } x+hg(x,a) \in \Omega_k,\; \forall x \in \Omega_k,\;
\forall a \in I_h(0),\; \forall h<h_0. \label{HIP2} \end{equation}
\begin{equation} \begin{array}{l} \Omega_k \rightarrow \Omega \text{ when } k \rightarrow 0 \text{, in the
following sense: } \Omega_k \subset \Omega \text{ and} \\
\forall K \text{ compact contained in } \Omega,\; \exists \bar{k}(K)
/ K \subset \Omega_k \quad \forall k \le \bar{k}(K).
\end{array} \label{HIP3}
\end{equation}
If $d_i$ is the diameter of the simplex $S_i^k$, then $\exists
\chi_1 > 0$ such that for every simplex of $\Omega_k$, there exists
a sphere of radius
    \begin{equation}
    r \le \chi_1 d_i. \label{HIP4}
    \end{equation}
    in the interior of the simplex. Furthermore there exists $M$, independent of the discretization,
such that
\begin{equation}
\frac{k}{d_i} \le M \qquad \forall i. \label{HIP5}
\end{equation}
We consider the set $W_k$ of functions $w: \Omega_k \times I_h
\rightarrow \RR,\; w(\cdot,a)$ continuous in $\Omega_k$, with
$\frac{\partial w(\cdot,a)}{\partial x}$ constant in the interior of
each simplex of $\Omega_k$ (that is, $w$ is a linear finite
element).

Let $V^k$ be the union over $j$ of the vertices of the family of simplices $\{S_j^k\}$. If $N$ is the cardinality of $V^k$ we consider an appropiate ordering: $V^k=\{ x^i,\; i=1,\ldots,N\}$.
\begin{obs} It is obvious that any $w \in W_k$ is completely determined by the values $w(x^i,a),\, i=1,\ldots,N,\, a \in I_h$.
\end{obs}
\begin{obs} In order to accomplish simplicity in the notation and clarity in the arguments, we will use the letters $C, M, K$ to denote arbitrary constants (whose values depend on the context where they appear and on the data of the problem, the constants $\lambda, M_g, M_f, L_g, L_f$, etc.) but not on the parameters of discretization $h, k$ or of regularization $\rho$, etc.
\end{obs}
We define the operator $A_k^h$ the following way
\[ A_k^h: W_k \rightarrow W_k \]
\begin{equation}
(A_k^h w)(x^i,a) = \min_{b \in I_h(a)} \left\{ (1-\lambda h) w(x^i + hg(x^i,a),b) + hf(x^i,a) \right\},\quad \forall x^i \in V^k,\, \forall a \in I_h. \label{56}
\end{equation}
\begin{obs}
The operator $A_k^h$ is merely the restriction of $A^h$ to the space $W_k$.
\end{obs}
The final problem to be resolved numerically is the following
\begin{equation}
\textsl{Problem } P_k: \text{ Find the fixed point of the operator }
A_k^h \label{57}
\end{equation}
It is clear that $A_k^h$ is a contractive operator in $W^k$ with the norm
\[ ||w|| = \max_{i,j} |w(x^i,a^j)| \]
that is $\forall w, \bar{w} \in W_k$
\begin{equation}
||A_k^h(w) - A_k^h(\bar{w})|| \le (1-\lambda h) ||w-\bar{w}|| \label{58}
\end{equation}
and therefore it has an unique fixed point, say $u_k^h$. Then
\[
u_k^h = A_k^h(u_k^h).
\]
\begin{obs} The problem $P_k$ is a nonlinear fixed point problem with a very special structure. It is equivalent to a stochastic control problem over a Markov chain (the reason for this relationship is the existence of a {\em discrete maximum principle} (see \cite{33,34,45}), valid for the schemes of discretization used in the definitions of $A_k^h$). In particular, it can be easilly seen that it represents a finite sequence of concatenated optimal stopping times. Therefore it can be solved by applying specially designed methods for those problems (Picard or Howard, see \cite{45}, \cite{77}).\\
\end{obs}
The main result of this paper is the
$\left(h+\frac{k}{\sqrt{h}}\right)^\gamma-$ convergence of the fully
discretized infinite horizon problem, where $\gamma$ is the H\"older
constant of the value function $u$, and is stated in Theorem
\ref{main}.

The paper is organized as follows. In section 2 we introduce the
associated finite horizon problem, its discretization in time and
the fully discretized problem. In section 3 we state the main result
of this paper, namely the convergence of the fully discrete scheme
in the infinite horizon case and state and prove some necessary
results. In section 4 we prove the main result and make some remarks
about the choice of the discretization parameters. In
section 5 we show some numerical implementations. Finally, in section 6 we state the conclussions.

\section{The associated finite horizon problem}

With the purpose of obtaining some technical results on convergence,
we consider a similar problem whose main difference is that it deals
with finite horizon. The problem consists in finding the value
function $u_T$, defined in the following way:
 $\forall t \in [0,T],\, \forall x \in \Omega,\,
\forall a \in [0,1]$,
\begin{equation}
u_T(t,x,a):= \inf_{a \in \calA_T(a)} J_T(t,x,\alpha(\cdot)),
\label{defu_T}
\end{equation}
where $\calA_T(a)$ is the set of non-decreasing functions defined in
$[t,T)$ with values in $[a,1]$, and the functional $J_T$ is defined
by
\begin{equation}
J_T(t,x,\alpha(\cdot)):= \int\limits_t^T
f(y(s),\alpha(s))e^{-\lambda (s-t)} \text{d}s \label{J_T}
\end{equation}
with $y(\cdot)$ being the solution of \eqref{dinam} with initial
condition $y(t)=x$. \\
If conditions \eqref{condg} and \eqref{condf} hold, then
\[
| u_T(t,x,a) |  \le  \frac{M_f}{\lambda} (1-e^{\lambda t})
\]and
\[
| u_T(t,x,a) - u_T(t,\bar{x},\bar a) | \le L_T \left(\|x-\bar{x}\| +
|a-\bar a| \right)
\]
where $L_T=\frac{L_f}{\lambda - L_g}$ if $\lambda > L_g$,
$L_T=\frac{L_f}{L_g - \lambda}e^{(L_g-\lambda)T}$ if $\lambda <
L_g$, and $L_T=TL_f$  if $\lambda = L_g$. Also, there exists a
positive constant $M_T$ such that
\[
| u_T(t,x,a) - u_T(\bar{t},x,a) | \le M_T|t-\bar{t}|.
\]
The proof of these properties makes use of classic techniques. \\
For the finite horizon case, the associated HJB equation takes the
form
\begin{equation}
\min \left( Lu_T(t,x,a),\frac{\partial u_T(t,x,a)}{\partial a}
\right) = 0 \qquad \text{in } (0,T) \times \Omega \times (0,1),
\label{HJB2}
\end{equation}
where $$ Lu_T = \frac{\partial u_T}{\partial t} + \frac{\partial
u_T}{\partial x} g + f - \lambda u_T,$$ with final condition
\[
u_T(T,x,a) = 0, \qquad \forall (x,a) \in \Omega \times [0,1],
\]
and boundary condition
\begin{equation}
u_T(t,x,1) = \int\limits_t^T f(\eta(s),1)e^{-\lambda (s-t)}
\text{d}s, \label{condfr}
\end{equation}
where $\eta(s)$ is the trajectory with initial value $\eta(t)=x$
corresponding to the control $\alpha \equiv 1$.

The viscosity solution of the equation \eqref{HJB2} is defined
similarly to the infinite horizon case  (see \cite{37}), proving that the value function $u_T$ is the unique
viscosity solution. \\
A different approach to the finite horizon problem with monotone
controls has been made in \cite{Hellwig:08} establishing a
Pontryagin Maximum Principle.

In a similar way to the infinite horizon problem, we consider for
the finite horizon case a discretization in time.  Let $h>0$ and
$I_h$, $I_h(a)$ as in the infinite horizon case.
\begin{obs} In what follows, we consider that $h^{-1}$ is an integer and that the horizon $T=\mu h $, with $\mu$ an integer.
\end{obs}
In order to find the solution of \eqref{HJB2} we employ recursive
approximation schemes. A natural discretization (by finite
differences) of the HJB equation gives the scheme
\begin{equation}
\left\{ \begin{array}{rcl}
u_T^h(\mu,x,a)&=&0. \quad \forall x \in \Omega,\, \forall a \in I_h, \vspace{2mm} \\
u_T^h(n-1,x,a) &=& \min\limits_{b \in I_h(a)} \left\{ (1-\lambda h)
u_T^h(n,x+hg(x,a),b) + hf(x,a) \right\},\vspace{2mm} \\ &&\qquad
\qquad\qquad \qquad\qquad\text{for }n =1,\ldots,\mu.
\end{array} \right. \label{defu_T^h}
\end{equation}
In \cite{APP} we proved the following
\begin{lem} \label{LemauT} \cite[Lemma 4.1]{APP} Under the hypotheses \eqref{condg} and \eqref{condf}, we have
\begin{equation}
|u(x,a)-u_T(t,x,a)| \le \frac{M_f}{\lambda} e^{-\lambda(T-t)}.
\label{40}
\end{equation} \end{lem}
and the
\begin{teo}\label{TeoPhi} \cite[Theorem 3.5]{APP}
\begin{equation}
\left| u_T(nh,x,a)-u_T^h(n,x,a) \right| \le C \phi(n) h, \label{38}
\end{equation}
where $\phi(n)$ is defined by
\[
\phi(n)= \left\{ \begin{array}{ll} e^{(L_g-\lambda)T+\lambda n h} \quad & \text{if } L_g > \lambda, \vspace{2mm} \\
Te^{L_g nh} \quad & \text{if } L_g = \lambda, \vspace{2mm} \\
e^{L_g nh} \quad & \text{if } L_g < \lambda
\end{array} \right.
\]
and $C>0$ is a constant independent of $h$.
\end{teo}
We define now $u_{k,T}^h$, the function of optimal cost totally discrete for the problem with finite horizon, where
\[ u_{k,T}^h(n,\cdot,\cdot) \in W_k \qquad \forall n=0,\ldots,\mu \]
is determined by the recursive scheme
\[
\left\{ \begin{array}{rcl} u_{k,T}^h(n-1,\cdot,\cdot) & = & (A_k^h u_{k,T}^h)(n,\cdot,\cdot) \qquad n=1,\ldots,\mu \vspace{5mm} \\
u_{k,T}^h(\mu,\cdot,\cdot) \equiv 0, \end{array} \right.
\]
for $0 \le n \le \mu$. Here, the operator $A_k^h$ is the one defined in \eqref{56}.
\section{Main Result. Convergence of the totally discrete problem}
The central result of this section gives a general estimation of the
error with respect to the parameters of discretization $h$ and $k$
and the data of the problem.
\begin{teo} \label{main} Let us suppose that \eqref{condg}, \eqref{condf} and the general properties of the triangulation hold, detailed in section 1. Thus there exists a constant $M$ independent of $h, k$ such that
\begin{equation}
|u(x,a) - u_k^h(x,a)| \le \left\{ \begin{array}{lll} M \left( h + \frac{k}{\sqrt{h}} \right), \quad & \quad & \text{if } L_g < \lambda \vspace{3mm} \\
M \left( h + \frac{k}{\sqrt{h}} \right)^\gamma, \quad & \text{with }
\gamma = \frac{\lambda}{L_g} \quad & \text{if } L_g > \lambda
\vspace{3mm} \\ \label{62} M \left( h + \frac{k}{\sqrt{h}}
\right)^\gamma, \quad & \text{with } \gamma \in (0,1) \quad &
\text{if } L_g = \lambda \end{array} \right.
\end{equation}
\end{teo}
The procedure used to obtain the estimation of $|u(x,a) - u_k^h(x,a)|$ is based in the decomposition in the following terms
\begin{equation} \begin{array}{rcl}  |u(x,a) - u_k^h(x,a)| &\le &|u(x,a) - u_T(0,x,a)| + |u_T(0,x,a) - u_T^h(0,x,a)| \vspace{3mm}\\
& & + |u_T^h(0,x,a) - u_{k,T}^h(0,x,a)| + |u_{k,T}^h(0,x,a) - u_k^h(x,a)|. \label{63} \end{array} \end{equation}
\subsection{Preliminary results}
\begin{note} For the proofs of the next two Lemmas, in order to simplify the exposition of the arguments, we will work only with the case $\Omega = \RR^\nu$. The general case can be treated without fundamental changes using the perturbation techniques of the domain described in \cite{54}.
\end{note}
\begin{lem} If \eqref{condg}, \eqref{condf} and \eqref{HIP2} are satisfied, the following holds
\begin{equation}
|u_{k,T}^h(0,x,a) - u_k^h(x,a)| \le \frac{M_f}{\lambda} e^{-\lambda T} \label{64}
\end{equation}
\end{lem}
\begin{proof} $A_k^h$ is a contractive operator, whose fixed point is $u_k^h$, that is,
\[ u_k^h = A_k^h u_k^h, \]
$u_{k,T}^h$ is also defined in terms of $A_k^h$ by
\[ u_{k,T}^h = (A_k^h)^\mu (w_T), \]
being $w_T=0$. By virtue of \eqref{58}, we have
\[ |u_{k,T}^h (0,x,a) - u_k^h(x,a)| \le (1-\lambda h)^\mu |u_k^h|, \]
then
\begin{equation}
|u_{k,T}^h(0,x,a) - u_k^h(x,a)| \le \frac{M_f}{\lambda} e^{-\lambda T}. \label{74}
\end{equation}
\end{proof}
\begin{lem} Under the hypotheses \eqref{condg}, \eqref{condf}, \eqref{HIP2}, \eqref{HIP4} and \eqref{HIP5}, we have
\[ \max_{0 \le n \le \mu} |u_T^h(n,x,a) - u_{k,T}^h(n,x,a)| \le M \phi(T)\frac{k}{\sqrt{h}}, \]
with
\begin{equation}
\phi(T)= \left\{ \begin{array}{ll} 1 \qquad & \text{if } L_g < \lambda \vspace{3mm} \\
e^{(L_g-\lambda)T} \qquad & \text{if } L_g > \lambda \vspace{3mm} \\
T \qquad & \text{if } L_g = \lambda. \end{array} \label {65} \right.
\end{equation}
\end{lem}
To prove this lemma we will need some auxiliary results, which will be obtained next. It will be necessary to decompose $$|u_T^h(n,x,a) - u_{k,T}^h(n,x,a)|$$ in three different addends by intercalating specially defined regularized functions using a regular function and the convolution operator. In the following propositions we will seize the regularity properties of these functions.
\subsubsection*{Definition of $u_{T,\rho}^h$, regularized function of $u_T^h$}
\begin{defi} We consider $\beta(\cdot) \in C^{\infty}(\RR^\nu)$ such that
\[ \beta(x) \ge 0\; \forall x,\; \text{support of } \beta \subset B_1=\{ x \in \RR^\nu\,/\, ||x|| \le 1 \},\, \int_{\RR^\nu} \beta(x) \text{d}x = 1. \]
We define $\forall \rho \in \RR^+$
\[
\beta_\rho(x)=\frac{1}{\rho^\nu} \beta(\frac{x}{\rho}) \ge 0.
\]
We obtain a regular approximation of $u_T^h$ by considering the convolution (with respect to the spatial variables) with the function $\beta_\rho(\cdot)$, that is $\forall n=0,\ldots,\mu,\; \forall a \in I_h,\; x \in \Omega$
\begin{equation}
u_{t,\rho}^h(n,x,a) = \int_{B(\rho)} u_T^h(n,x-y,a)\beta_\rho(y) \text{d}y = \left( u_T^h(n,\cdot,a) \ast \beta_\rho \right)(x). \label{77}
\end{equation}
\end{defi}
\subsubsection*{Definition of $\tilde{u}_{T,\rho}^h$, linear interpolate of $u_{T,\rho}^h$}
\begin{defi} For each $n=0,1,\ldots,\mu,$ we define $\tilde{u}_{T,\rho}^h(n,\cdot,\cdot) \in W_k$ by assigning in the nodes of interpolation in a natural way the values
\begin{equation}
\tilde{u}_{T,\rho}^h(n,x^i,a) = u_{T,\rho}^h(n,x^i,a), \qquad \forall n = 0,\ldots,\mu,\; \forall i = 1,\ldots,N, \; \forall a \in I_h. \label{78}
\end{equation}
\end{defi}
\begin{propo1} For $u_T^h$ the following estimate holds
\[
|u_T^h(n,x,a) - u_{T,\rho}^h(n,x,a)| \le L_u \rho.
\]
\end{propo1}
\begin{proof}
by the definition of convolution we have from \eqref{77}
\[ |u_T^h(n,x,a) - u_{T,\rho}^h(n,x,a)| \le \int_{B(\rho)} |u_T^h(n,x,a) - u_T^h(n,x-y,a)| \beta_\rho(y) \text{d}y; \]
since $u_T^h$ is Lipschitz continuous, we have
\[ |u_T^h(n,x,a) - u_{T,\rho}^h(n,x,a)| \le L_u \int_{B(\rho)} ||y|| \beta_\rho(y) \text{d}y \le L_u \rho. \]
\end{proof}
\begin{note} The functions $u_T^h,u_{T,\rho}^h$ and the operator $A^{h,b}$ verify the inequalities
\begin{equation}
u_T^h(n,x,a) \le \left(A^{h,b} u_T^h(n+1,\cdot,\cdot) \right)(x,a) \qquad \forall b \in I_h(a), \label{80}
\end{equation}
\begin{equation}
u_{T,\rho}^h (n,x,a) \le \left( \left( A^{h,b} u_T^h(n+1,\cdot,\cdot) \right)(\cdot,a) \ast \beta_\rho \right)(x) \qquad \forall b \in I_h(a), \label{81}
\end{equation}
where
\[ \left( A^{h,b} u_T^h(n+1,\cdot,\cdot) \right)(x,a) = (1-\lambda h) u_T^h(n+1,x+hg(x,a),b)+hf(x,a). \]
\end{note}
\begin{propo2} Every Lipschitz continuous function $w \in W^{1,\infty}(\Omega)$ verifies (being $L_w$ its Lipschitz constant)
\begin{equation}
\left| (A^{h,b} w) \ast \beta_\rho - A^{h,b}(w \ast \beta_\rho) \right| \le (1-\lambda h) L_w L_g \rho h + L_f \rho h. \label{82}
\end{equation}
Besides
\begin{equation}
\left| \left(A^{h,b} u_{T,\rho}^h(n,\cdot,\cdot)\right)(x,a) - \left(A^{h,b} \tilde{u}_{T,\rho}^h(n,\cdot,\cdot)\right)(x,a) \right| \le C(1-\lambda h) L_u \frac{k^2}{\rho}. \label{83}
\end{equation}
\end{propo2}
\begin{proof}
Let us see first that \eqref{82} is valid:
\[ {\small\begin{array}{rcl} \left( (A^{h,b} w) \ast \beta_\rho - A^{h,b}(w \ast \beta_{\rho}) \right)(x,a) & = & \int\limits_{B(\rho)} \left((1-\lambda h)w(x-\eta+hg(x-\eta,a),b)+hf(x-\eta,a) \right) \beta_\rho(\eta) \text{d}\eta \vspace{3mm} \\
& & -(1-\lambda h) \int\limits_{B(\rho)} w(x-\eta+hg(x,a),b) \beta_\rho(\eta) \text{d}\eta - hf(x,a), \end{array} }\]
then
\[ \left|\left( (A^{h,b} w) \ast \beta_\rho - A^{h,b}(w \ast \beta_\rho) \right) \right| \le \]
\[ \begin{array}{ll} \le & \int\limits_{B(\rho)} (1-\lambda h) \left| w(x-\eta+hg(x-\eta,a)b) - w(x-\eta +hg(x,a),b) \right| \beta_\rho(\eta) \text{d} \eta \vspace{3mm} \\
& + h\int\limits_{B(\rho)} \left| f(x-\eta,a) - f(x,a) \right| \beta_\rho(\eta) \text{d}\eta \le (1-\lambda h) L_w L_g h \rho + L_f h \rho, \end{array} \]
inequality that clearly implies \eqref{82}. \\
The function $u_{T,\rho}^h$ has second derivatives bounded by
\begin{equation}
||D^2 u_{T,\rho}^h|| \le \tilde{C} \frac{L_u}{\rho}, \label{84}
\end{equation}
where $\tilde{C}$ is a constant that only depends on $\beta(\cdot)$, since $u_{T,\rho}^h$ is the regularization if a Lipschitz continuous function. In consequence, from \eqref{HIP4} and \eqref{84}, the difference between $u_{T,\rho}^h$ and its linear interpolate $\tilde{u}_{T,\rho}^h$ is bounded by $C L_u \frac{k^2}{\rho}$.
\[ \left| \left( A^{h,b} u_{T,\rho}^h(n,\cdot,\cdot) \right)(x,a) - \left( A^{h,b} \tilde{u}_{T,\rho}^h(n,\cdot,\cdot) \right)(x,a) \right| \le (1-\lambda h) \left|(u_{T,\rho}^h-\tilde{u}_{T,\rho}^h)(n,x+hg(x,a),b) \right| \le \]
\[ \le C(1-\lambda h) L_u \frac{k^2}{\rho} \]
and \eqref{83} is now proved.
\end{proof}
\begin{propo3}
The following estimate holds
\begin{equation}
\tilde{u}_{T,\rho}^h(n,x,a) - u_{k,T}^h(n,x,a) \le \frac{1-\lambda h}{\lambda} L_u L_g \rho + \frac{L_f}{\lambda} \rho + \frac{1-\lambda h}{\lambda h} C_1 \frac{L_u k^2}{\rho}. \label{85}
\end{equation}
\end{propo3}
\begin{proof} Let
\begin{equation}
E_n = \sup_{x \in \Omega_k,\, a \in I_h} \left( \tilde{u}^h_{T,\rho}(n,x,a) - u_{k,T}^h(n,x,a) \right) = \max_{i=1,\ldots,N,\, a \in I_h} \left(\tilde{u}_{T,\rho}^h(n,x^i,a) - u_{k,T}^h(n,x^i,a) \right). \label{86}
\end{equation}
From the properties \eqref{81} and \eqref{82}, we have $\forall b \in I_h(a)$
\[ u_{T,\rho}^h(n-1,x,a) = \left( \left( A^h u_T^h(n,\cdot,a) \right) \ast \beta_\rho \right)(x) \le \left( \left(A^{h,b} u_T^h(n,\cdot,a) \right) \ast \beta_\rho \right)(x) \le \]
\begin{equation}
\le \left(A^{h,b}\left( u_T^h(n,\cdot,a) \ast \beta_\rho \right) \right)(x) + \phi(x,h,\rho) = \left( A^{h,b} u_{T,\rho}^h(n,\cdot,a) \right)(x) + \phi(x,h,\rho), \label{87}
\end{equation}
with
\[ |\phi(x,h,\rho)| \le (1-\lambda h) L_u L_g h \rho + L_f h \rho. \]
From \eqref{56}, if $w \in W^k$, we have
\begin{equation}
\left( A_k^h w \right) (x,a) = \min_{b \in I_h(a)} \left\{ (1-\lambda h) w(x+hg(x,a),b) +hf(x,a) \right\} = \left(A^h \tilde{w} \right)(x,a) \le \left(A^{h,b} \tilde{w} \right)(x,a). \label{88}
\end{equation}
From \eqref{83},\eqref{87} and \eqref{88} we have
\[
u_{T,\rho}^h(n-1,x^i,a) \le \left(A^{h,b}
\tilde{u}_{T,\rho}^h(n,x,a)\right)(x^i) + \phi(x^i,h,\rho) +
\varphi(x^i,h,\rho,k),
\]
then, $\forall b \in I_h(a)$, the following is valid
\[ u_{T,\rho}^h(n-1,x^i,a) \le A^{h,b} \tilde{u}^h_{T,\rho}(n,\cdot,a) + \phi(x^i,h,\rho) + \varphi(x^i,h,\rho,k) \]
where
\[ |\varphi(xî,h,\rho,k)| \le C(1-\lambda h) \frac{k^2}{\rho} \]
and
\[ u_{k,T}^h(n-1,x^i,a) = \left( A_k^h u_k^h \right) (n,x^i,a) = \left( A^{h,\bar{a}} u_{k,T}^h \right)(n,x^i,a), \]
where $\bar{a}$ makes the minimum of \eqref{88} for $u_{k,T}^h$, then
\[ u^h_{T,\rho}(n-1,x^i,a) - u_{k,T}^h(n-1,x^i,a) \le \left(A^{h,\bar{a}} \tilde{u}_{T,\rho}^h \right)(n,x^i,a) - \left(A^{h,\bar{a}} u_{k,T}^h \right) (n,x^i,a) + \phi(x^i,h,\rho) + \varphi(x^i,h,\rho,k). \]
From \eqref{78}, we have
\[ \tilde{u}^h_{T,\rho}(n-1,x^i,a) - u_{k,T}^h(n-1,x^i,a) \le (1-\lambda h) \left( \tilde{u}^h_{t,\rho}(n,x^i,a) - u_{k,T}^h(n,x^i,a) \right) + \phi(x^i,h,\rho) + \varphi(x^i,h,\rho,k). \]
In consequence, taking into account the definition \eqref{86}, we have
\begin{equation}
E_{n-1} \le (1- \lambda h)E_n + (1- \lambda h) L_u L_g h \rho + L_f h \rho + C(1-\lambda h) L_u \frac{k^2}{\rho} \label{90}
\end{equation}
Finally, since
\[ E_{n-1} \le \delta E_n + B \; \text{ implies } \; E_0 \le \delta^n E_n + \frac{B}{1-\delta}, \]
replacing this estimate in \eqref{90} and taking into account that $E_\mu=0$, we obtain \eqref{85}, valid for every $0 \le n \le \mu,\; x \in V_k,\; a \in I_h$.
\end{proof}
\begin{propo4} The following estimate holds
\begin{equation}
u_{k,T}^h(n,x,a) - \tilde{u}_{T,\rho}^h(n,x,a) \le \frac{1-\lambda h}{\lambda} L_u L_g \rho + \frac{L_f}{\lambda} \rho + \frac{1-\lambda h}{\lambda h} C_1 \frac{L_u k^2}{\rho} + 2 L_u \rho. \label{91}
\end{equation}
\end{propo4}
The proof of this proposition is completely similar to the previous
one and we shall omit it. We are now in position to prove Lemma 3.2.
\subsubsection*{Proof of Lemma 3.2} \begin{proof} For each $ 0 \le n \le \mu$, we will show
\[ \max_{0 \le n \le \mu} |u_T^h(n,x,a)-u_{k,T}^h(n,x,a)| \le M \phi(T) \frac{k}{\sqrt{h}}. \]
We start by decomposing the following way:
\begin{equation}
\begin{array}{rcl} |u^h_T(n,x,a) - u_{k,T}^h(n,x,a)| & \le & |u^h_T(n,x,a) - u_{T,\rho}^h(n,x,a)| + |u^h_{T,\rho}(n,x,a) - \tilde{u}_{T,\rho}^h(n,x,a)| \vspace{3mm} \\ \label{92}
& & |\tilde{u}^h_{T,\rho}(n,x,a) - u_{k,T}^h(n,x,a)|, \end{array}
\end{equation}
applying the previous propositions, \eqref{85} and \eqref{91}, we can estimate the terms of the right side of \eqref{92} and obtain
\begin{equation}
|u^h_T(n,x,a) - u_{k,T}^h(n,x,a)| \le L_u \rho + C \frac{L_u k^2}{\rho} + \frac{1}{\lambda} L_u L_g \rho + \frac{L_f}{\lambda} \rho + \frac{1}{\lambda h} C_1 \frac{L_u k^2}{\rho} + 2 L_u \rho. \label{3_20}
\end{equation}
\subsubsection*{Analysis of the different cases}
\begin{itemize}
\item $L_g > \lambda$. \\
In this case, $L_u$ has the form
\[ L_u = L_f \frac{1}{L_g-\lambda} e^{(L_g - \lambda)T} \]
and the inequality \eqref{3_20} becomes
\[
|u_T^h(n,x,a) - u^h_{k,T}(n,x,a)| \le M e^{(L_g - \lambda)T} \left( \rho + \frac{k^2}{h \rho} \right),
\]
where
\[ M=\max \left( 3+\frac{L_g}{\lambda} + \frac{L_g -\lambda}{\lambda}, C + \frac{C_1}{\lambda} \right) \frac{L_f}{L_g - \lambda}. \]
Minimizing in $\rho$ the expression $(\rho - \frac{k^2}{h \rho})$, we have that the minimum of \eqref{92} is attained when $\rho = \frac{k}{\sqrt{h}}$, then
\[ |u^h_T(n,x,a) - u_{k,T}^h(n,x,a) | \le 2M e^{(L_g-\lambda)T} \frac{k}{\sqrt{h}}. \]
\item $L_g < \lambda$. \\
In this case $L_u= \frac{L_f}{\lambda - L_g}$ and the inequality \eqref{91} becomes
\[
|u_T^h(n,x,a) - u^h_{k,T}(n,x,a)| \le M \left( \rho + \frac{k^2}{h \rho} \right),
\]
where
\[ M=\max \left( 3+\frac{L_g}{\lambda} + \lambda - L_g, \frac{C}{\lambda} + \frac{C_1}{\lambda} \right) \frac{L_f}{\lambda - \L_g}. \]
Finally, in the same way as in the last case, we get
\[ |u^h_T(n,x,a) - u_{k,T}^h(n,x,a) | \le 2M \frac{k}{\sqrt{h}}. \]
\item $L_g = \lambda$. \\
In this case $L_u=L_f T$ and the inequality \eqref{91} becomes
\[ |u_T^h(n,x,a) - u_{k,T}^h(n,x,a)| \le M T \left( \rho + \frac{k^2}{h \rho} \right), \]
where
\[ M = \max (4,C + \frac{C_1}{\lambda}) L_f. \]
Then
\[ |u_T^h(n,x,a) - u_{k,T}^h(n,x,a)| \le 2 M T \frac{k}{\sqrt{h}}. \]
\end{itemize}
\end{proof}
\section{Proof of Theorem 3.1}
Having proved this two Lemmas we proceed now to prove the Theorem
3.1.
\begin{proof} We will use mainly the estimates given by the previous lemmas. From \eqref{40}, \eqref{64} and \eqref{65} we have
\begin{equation}
\begin{array}{rcl} |u(x,a) - u_k^h(x,a)| & \le & |u(x,a) - u_T(0,x,a)| + |u_T(0,x,a) - u_T^h(0,x,a)| + \vspace{3mm} \\
& & + \max\limits_n |u_T^h(n,x,a) - u_{k,T}^h(n,x,a)| +|u_{k,T}^h(0,x,a) - u_k^h(x,a)| \vspace{3mm} \\ \label{66}
& \le & \frac{M_f}{\lambda} e^{-\lambda T} + Che^{(L_g-\lambda) T} + M\phi(T)\frac{k}{\sqrt{h}} + \frac{M_f}{\lambda} e^{-\lambda T}. \end{array}
\end{equation}
By virtue of the lemmas 2.1, 3.1, 3.2, the analysis of \eqref{63} can be divided into the three following cases:
\begin{itemize}
\item $L_g > \lambda$. \\
\begin{equation}
|u(x,a) - u_k^h(x,a)| \le M_1 \left( e^{-\lambda T} + e^{(L_g - \lambda)T} \left( h + \frac{k}{\sqrt{h}} \right) \right), \label{67}
\end{equation}
where
\[ M_1 = \max \left( 2 \frac{M_f}{\lambda},C,M \right). \]
Minimizing \eqref{67} in the $T$ variable, we obtain
\begin{equation}
|u(x,a) - u_k^h(x,a)| \le K \left( h+\frac{k}{\sqrt{h}} \right)^\gamma \label{68}
\end{equation}
where
\[ \gamma=\frac{\lambda}{L_g},\; K=M_1 \left( \left(\frac{1-\gamma}{\gamma} \right)^\gamma + \left(\frac{1-\gamma}{\gamma} \right)^{\gamma-1} \right). \]
The minimum in \eqref{68} refers to the instance where the minimum is taken in a point of the discrete set $\{ nh / n=0,1,\ldots\}$; for the general case, in \eqref{68}, we must add to the obtained bound a term of order $h$, which does not modify essentially the estimate.
\item $L_g < \lambda$. \\
In this case, by virtue of \eqref{66} we have
\[ |u(x,a) - u_k(x,a)| \le \frac{M_f}{\lambda} e^{-\lambda T} + C h + M \frac{k}{\sqrt{h}} + \frac{M_f}{\lambda} e^{-\lambda T}; \]
so, by taking limit $T \rightarrow + \infty$, we obtain
\begin{equation}
|u(x,a) - u_k^h(x,a)| \le M_1 \left( h + \frac{k}{\sqrt{h}} \right), \label{69}
\end{equation}
where $M_1=\max(C,M)$.
\item $L_g = \lambda$. \\
In this case, by virtue of \eqref{65}, \eqref{66}, we have
\[ |u(x,a) - u_k^h(x,a)| \le \frac{M_f}{\lambda} e^{-\lambda T} + ChT + M \frac{k}{\sqrt{h}} T +\frac{M_f}{\lambda} e^{-\lambda T} \]
then
\begin{equation}
|u(x,a)-u_k^h(x,a)| \le M_1 \left( e^{-\lambda T} + \left(h+\frac{k}{\sqrt{h}} \right) T \right) \label{70}
\end{equation}
where
\[ M_1 = \max \left( 2 \frac{M_f}{\lambda},C,M \right). \]
For $\frac{1}{\lambda}(h+\frac{k}{\sqrt{h}}) \le 1$, the minimum in the right side of \eqref{70} is assumed by
\[ \bar{T} = -\frac{1}{\lambda} \log \frac{1}{\lambda} \left( h + \frac{k}{\sqrt{h}} \right) ; \]
then, replacing in \eqref{70} we obtain:
\[ |u(x,a)-u_k^h(x,a)| \le M_1 \left( h+\frac{1}{\lambda}(h+\frac{k}{\sqrt{h}}) - \frac{1}{\lambda} (h + \frac{k}{\sqrt{h}}) \log ( \frac{1}{\lambda}(h+\frac{k}{\sqrt{h}})) \right). \]
Since the following property holds
\[ -x \log(x) \le K x^\gamma,\; \gamma \in (0,1), \; K = \frac{1}{1-\gamma} e^{-1}, \]
we have
\[
|u(x,a) - u_k^h(x,a)| \le C \left( h+ \frac{k}{\sqrt{h}} \right)^\gamma,
\]
where $C=\frac{M_1}{\lambda} \max (1,\frac{1}{1-\lambda} e^{-1}).$
\end{itemize}
\end{proof}
\subsection{About the choice of the discretization parameters}
Taking into account the central result \eqref{62} we can reach the
following conclusions:
\begin{itemize}
\item When $h$ is of order $k$, that is $c_1 k \le h \le c_2 k$ being $c_1$ and $c_2$ constants, the convergence is of order $k^{\frac{\gamma}{2}}$, that is
    \[ |u(x,a) - u_k^h(x,a) | \le C k^{\frac{\gamma}{2}}. \]
\item The best speed of convergence is attained selecting $h$ of order $k^\frac{2}{3}$, that is $c_1 k^\frac{2}{3} \le h \le c_2 k^\frac{2}{3}$. In this case we obtain the estimate
    \begin{equation}
    |u(x,a) - u_k^h(x,a) | \le C k^{\frac{2}{3}\gamma}. \label{72}
    \end{equation}
\item The estimate \eqref{72} is optimal, in the sense that it is possible to construct examples where
\[ |u(x,a) - u_k^h(x,a) | \ge c k^{\frac{2}{3}\gamma} \]
The construction of these examples can be made by following the techniques used in \cite{63}.
\end{itemize}
\section{Example}
We consider a simple example in $\mathbb{R}^2$. The controlled dynamics is:
\begin{equation}
(D)\left\{ \begin{array}{ll} \dot{y}(s)=g(y(s),\alpha(s)) & s>0 \vspace{2mm} \\
y(0)=x=(x_1,x_2) & \end{array} \right.
\end{equation}
where
\[ g(x_1,x_2,a)=(-(a+1)x_1,-(a+1)x_2) \]
and $\calA(a)$ is defined as in the general case and
$\Omega=(-1,1)^2$. The cost to minimize is given by
\[ u(x,a) = \inf_{\alpha \in \calA(a)} \int_0^\infty
f(y(s),\alpha(s))e^{-s} \text{d}s. \] Here
\[ f(x_1,x_2,a):= a\left(\frac{1}{4} -(x_1^2+x_2^2)\right). \]
It is clear that $g$ and $f$ verify the general hypotheses. Here, the discount factor is
$\lambda = 1$.

We present now a full discretization of the problem. We introduce the discretization parameter $h$ to define a equispaced discretization in the time and control spaces. The
control variable $a$ takes values in the set
\[
I_h := \left\{ ih | i=0 \dots \frac{1}{h} \right\}.
\]
We also define $I_h(a) = I_h \cap [a,1]$.
We proceed now to define the discretization of the state space. For this aim for each $k>0$ we introduce a family of finite elements in the following way: first we define the set of vertices
\begin{equation}
V_k:=\{ P_{i,j}(-1+ik,-1+jk) \,;\, i,j=1,\ldots,2 \tfrac{1}{k} +1 \}
\end{equation}
\begin{note} for the sake of simplicity we assume here that $\frac{1}{k}$ is an integer. \end{note}
These vertices define the family of triangles $S_{i,j}^k$ whose union over $i,j$ is
\[ \Omega_k=[-1+k,1-k]^2 \]

Let us check that this family of triangulations verify the hypotheses. Being $\text{diam }S_{i,j}^k=k$ for all $j$ it is clear that \eqref{HIP1} is verified. The functions $g^1,g^2$ are specially defined so that \eqref{HIP2} holds in both cases. From the definition of $\Omega_k$, \eqref{HIP3} is straightforward. Finally we observe that taking $\chi_1=r=1$, \eqref{HIP4} and \eqref{HIP5} are verified.

The space $W_k$ is the set of piecewise linear continuous functions defined in $\Omega_k \times I_h$. We recall that the fully discrete solution $u_k^h$ is the unique solution of the fixed point problem of the operator $A_k^h$.
Let $u_0 \in W_k$ be the zero function. We will consider it as the initial function in our fixed point problem and for every $n \in \mathbb{N}$ we note
\[ u_n = A_k^h(u_{n-1}). \]
We have the following estimate
\begin{equation} \| u_n - u_h^k \| \le \dfrac{1}{h} \|u_n - u_{n-1} \| \label{cotaptofijo} \end{equation}
For different choices of $h,k$ we obtain approximations to the solution $u_k^h$ by applying re\-peatedly the operator $A_k^h$ to the initial function $u_0 \in W_k$. Actually in the implementations that will be shown here we take always $h=k$. In each case the stop criterion is
\[ \|u_n - u_{n-1} \| \le h^2 \] since from \eqref{cotaptofijo} this estimate assures us that
\[ \| u_n - u_h^k \| \le h. \]
In the following table we show the numerical results.
\begin{table}[ht]
\centering 
\begin{tabular}{c |c c} 
& $(D)$ &  \\
\hline\hline 
h=k & Iterations & Time(secs)\\ [0.5ex] 
\hline
0.50 & 1 & 0.0168\\ 
0.40 & 1 & 0.0025\\
0.30 & 2 & 0.0120\\
0.20 & 3 & 0.0995\\
0.10 & 10 & 3.5145 \\
0.05 & 33 & 135.8481\\
0.02 & 126 & 13527.0552\\ [1ex] 
\hline 
\end{tabular}
\label{table:nonlin} 
\end{table}

The algorithm was implemented in \textsl{Scilab 5.4.1 - 64 bits} in a \textsl{PC} equipped with an \textsl{Intel Core i7-3770K} processor with \textsl{8 GB DD3 RAM}.

\section{Conclussions}
In this work mainly we have developed efficient procedures of numerical resolutions of the Hamilton-Jacobi-Bellman equation associated to the optimal control problem with monotone controls, making use of the particular structure of the original problem to develop the discretization schemes and obtain its estimates of the speed of convergence. \\
Essentially, the discretization employed is equivalent to the resolution of a finite family of concatenated stopping time problems. According to this point of view, it can be obtained by applying mechanically the results contained in \cite{29,63}, an estimate of the speed of convergence of order $k^{\frac{\gamma}{4}}$. This estimate is improved in this work analyzing directly the control problem rather than taking this way leading to non optimal bounds.\\
Finally, we can remark that the optimization of the discretization
parameters has been studied and it has been established that in
general the choice of a relationship $h \approx k^{\frac{2}{3}}$
allows us to get the best possible result.

\end{document}